\documentclass[reqno]{amsart}%
\usepackage{amsfonts}
\usepackage{amsmath}
\usepackage{fullpage}
\usepackage{amssymb}
\usepackage{graphicx}%
\providecommand{\U}[1]{\protect\rule{.1in}{.1in}}
\newtheorem{theorem}{Theorem}
\theoremstyle{plain}

\newtheorem{corollary}[theorem]{Corollary}

\newtheorem{lemma}[theorem]{Lemma}

\numberwithin{theorem}{section}
\numberwithin{equation}{section}
\begin{document}
\title[Pontryagin-Krein Theorem]{Pontryagin-Krein Theorem: Lomonosov's proof and related results}
\author{Edward Kissin}
\address{E. Kissin: STORM, London Metropolitan University, 166-220 Holloway Road,
London N7 8DB, Great Britain;}
\email{e.kissin@londonmet.ac.uk}
\author{Victor S. Shulman}
\address{V. S. Shulman: Department of Mathematics, Vologda State University, Vologda,
Russia; }
\email{shulman.victor80@gmail.com}
\author{Yurii V. Turovskii}
\address{Yu. V. Turovskii}
\email{yuri.turovskii@gmail.com}
\thanks{2010 \textit{Mathematics Subject Classification.} Primary 47A15, Secondary 47L10}
\thanks{This paper is in final form and no version of it will be submitted for
publication elsewhere.}
\keywords{invariant subspace, Pontryagin space, Krein space, indefinite metric}
\dedicatory{To the memory of our dear friend and colleague Victor Lomonosov}
\begin{abstract}
We discuss Lomonosov's proof of the Pontryagin-Krein Theorem on
invariant maximal non-positive subspaces, prove the refinement of one theorem
from \cite{OShT} on common fixed points for a group of fractional-linear maps
of operator ball and deduce its consequences. Some Burnside-type counterparts
of the Pontryagin-Krein Theorem are also considered.

\end{abstract}
\maketitle

\section{Introduction and preliminaries}

In 1944 L. S. Pontryagin, stimulated by actual problems of mechanics,  published his famous paper \cite{Pontr} where it was proved that if an operator $T$ is selfadjoint with respect to a scalar
product with finite number $k$ of negative squares then $T$ has invariant non-positive
subspace of dimension $k$. The importance of results of this kind for
stability of some mechanical problems was discovered by S. L. Sobolev in 1938,
who proved the existence of non-positive eigenvectors in the case $k=1$.

Before giving precise formulations we introduce some notations. By indefinite
metric space we mean a linear space $H$ supplied with a semilinear form
$[x,y]$ satisfying the following condition: $H$ can be decomposed in a direct
sum of two subspaces $H_{+}$, $H_{-}$ ($x=x_{+}+x_{-}$, for each $x\in H$) in
such a way that $H$ is a Hilbert space with respect to the form
\[
(x,y)=[x_{+},y_{+}]-[x_{-},y_{-}].
\]
The decomposition of this kind is not unique but the dimensions of the summands
and the topology on $H$ do not depend on the choice of the decomposition. We
assume in what follows that $\dim H_{+}\geq\dim H_{-}$. If $\dim
H_{-}=k<\infty$ then one says that $H$ is a \textit{Pontryagin space }$\Pi
_{k}$, otherwise $H$ is called a \textit{Krein space}.

A vector $x\in H$ is called \textit{positive} (\textit{non-negative},
\textit{negative}, \textit{non-positive}, \textit{neutral}) if
\[
\lbrack x,x]>0\text{ }(\text{respectively }[x,x]\geq0,\text{ } [x,x]<0,\text{
}[x,x]\leq0,\text{ }[x,x]=0).
\]
A subspace is \textit{positive} (\textit{non-negative, non-positive,negative,
neutral}) if its non-zero elements are positive (respectively non-negative,
non-positive, negative, neutral). For brevity we write MNPS for maximal
non-positive subspaces.

Subspaces $H_{1},H_{2}$ of $H$ form \textit{a dual pair} if $H_{1}$ is
positive, $H_{2}$ is negative and $H=H_{1}+H_{2}$.

Sometimes it is convenient to start with a Hilbert space $H$ decomposed in the
orthogonal sum of two subspaces $H=H_{+}\oplus H_{-}$ and to set
\[
\lbrack x,y]=(x_{+},y_{+})-(x_{-},y_{-}).
\]
Denoting by $P_{+}$ and $P_{-}$ the projections onto $H_{+}$ and $H_{-},$
respectively, set $J=P_{+}-P_{-}.$ Then one can write the relation between two
"scalar products" in the form
\[
\lbrack x,y]=(Jx,y)\text{ and }(x,y)=[Jx,y].
\]
This notation determines the standard terminology. A space with indefinite
metric is often called a $J$-space, a vector $x$ is $J$-orthogonal to a vector
$y$ if $[x,y]=0$. An operator $B$ (we consider only bounded linear operators)
on $H$ is called $J$-\textit{adjoint} to an operator $A$ if $[Ax,y]=[x,By]$,
for all $x,y\in H$; we write $B=A^{\sharp}$. If $A^{\sharp}=A$ then $A$ is
called $J$-selfadjoint; an equivalent condition is $[Ax,x]\in\mathbb{R}$, for
all $x\in H$. If $\operatorname{Im}([Ax,x])\geq0$ for all $x$, then $A$ is
called $J$-dissipative.

Furthermore, $A$ is $J$-\textit{unitary} if $A^{\sharp}=A^{-1}$ (equivalently,
$A$ is surjective and $[Ax,Ay]=[x,y]$, for $x,y\in H$); $A$ is $J$%
-\textit{expanding} if $[Ax,Ax]\geq\lbrack x,x]$, for all $x\in H$.

In 1949 I. S. Iohvidov \cite{Ioh} constructed an analogue of Caley transform
for indefinite metric spaces which allowed him to deduce from Pontryagin's
Theorem the existence of an invariant MNSP for $J$-unitary operators on
$\Pi_{k}$-spaces. Then M. G. Krein \cite{Kr}, using absolutely different
approach, proved that a $J$-unitary (and, more generally, $J$-expanding)
operator $U$ in arbitrary indefinite metric space has an invariant MNPS, if
its "corner" $P_{-}UP_{+}$ is compact. Clearly, this condition holds in
$\Pi_{k}$-spaces. In 1964 Ky Fan \cite{KF} extended Krein's Theorem to
operators on Banach spaces preserving indefinite norms $\nu(x)=\Vert
(1-P)x\Vert-\Vert Px\Vert$ where $P$ is a projection of finite rank.

Now we have the following \textit{Pontryagin-Krein Theorem} (hereafter
\textit{PK-Theorem}).

\begin{theorem}
\label{PK} Let an operator $A$ on a Krein space $H$ be $J$-dissipative and let
$P_{+}AP_{-}$ be compact. Then there exists an MNPS invariant for $A$.
\end{theorem}

Note that the proof of Pontryagin's result in \cite{Pontr} was
very complicated and long. The Krein's proof in \cite{Kr} was short but far from
elementary, because it was based on the Schauder-Tichonov fixed-point Theorem. Moreover
 Ky Fan, to prove his version of the PK-Theorem, previously obtained a more
general fixed point theorem. We add that to deduce the result for $J$-dissipative
operators from the Krein's theorem about $J$-expanding operators, one needs to use
Iohvidov's Theory of Caley transformation for Krein spaces which is also very non-trivial.

In 1986 Victor Lomonosov in a talk at the Voronezh Winter School presented a
proof of Theorem \ref{PK} which was extremely short and completely elementary; this proof was published in \cite{Lom}. In Section 2 of our paper we present the
Lomonosov's proof in a complete form including the consideration of the finite-dimensional case. In Section 3 we consider the approach based on some fixed point theorems and discuss several results obtained on this way. In Section 4 we prove Theorem 4.1 which refines a theorem of M. Ostrovskii, V.S. Shulman and L. Turowska
\cite{OShT} about common fixed points for a group of fractional-linear maps of
operator ball. This allows us to estimate the similarity degree for a bounded representation of a group on a Hilbert space which preserves a quadratic form with finite number of negative squares. In Section 5  we prove by using Theorem 4.1 that any bounded quasi-positive definite function on a group is a difference
of two positive definite functions (this was known earlier only for amenable
groups). In the final section we discuss Burnside type counterparts of PK-Theorem.

\section{ Lomonosov's proof of PK-Theorem}

As usual, $\mathcal{B}(H_{1},H_{2})$ is the space of all bounded linear
operators from $H_{1}$ to $H_{2}$, and $\mathcal{B}(H) = \mathcal{B}(H,H)$ is the algebra of
all bounded linear operators on $H$. To any operator
$W:H_{-}\rightarrow H_{+}$ there corresponds the graph-subspace $L_{W}%
=\{x+Wx:x\in H_{-}\}$; it is easy to see that $L_{W}$ is maximal non-positive
if and only if $W$ is contractive, that is $\Vert W\Vert\leq1$. Conversely,
each MNPS is of the form $L_{W}$, for some contraction $W\in\mathcal{B}%
(H_{-},H_{+})$. It is not difficult to check that $L_{W}$ is invariant under
an operator $A\in\mathcal{B}(H)$ if and only if
\begin{equation}
WA_{11}+WA_{12}W-A_{21}-A_{22}W=0, \label{invar}%
\end{equation}
where
\begin{equation}
A_{11}=P_{-}AP_{-},\;A_{12}=P_{-}AP_{+},\;A_{21}=P_{+}AP_{-},\;A_{22}%
=P_{+}AP_{+}. \label{matr}%
\end{equation}
Lomonosov in {\cite{Lom}} introduced a "mixed" convergence $(M$%
-\textit{convergence}$)$ in $\mathcal{B}(H)$: a sequence $\{A^{(k)}%
\}_{k=1}^{\infty}$ of operators $M$-converges to an operator $A$, if
$A_{11}^{(k)}\rightarrow A_{11}$ and $(A_{22}^{(k)})^{\ast}\rightarrow
(A_{22})^{\ast}$ in the strong operator topology (SOT), $A_{21}^{(k)}%
\rightarrow A_{21}$ in the weak operator topology (WOT) and $A_{12}%
^{(k)}\rightarrow A_{12}$ in norm.

\begin{theorem}
\emph{\cite{Lom}}\label{stable} Let a sequence $\{A^{(k)}\}_{k=1}^{\infty}$ of
operators $M$-converge to an operator $A$. If each $A^{(k)}$ has an MNPS then
$A$ has an MNPS.
\end{theorem}

\begin{proof}
It follows from our assumptions, that for each $k$, there is a contraction
$W_{k}\in\mathcal{B}(H_{-},H_{+})$ satisfying
\begin{equation}
W_{k}A_{11}^{(k)}+W_{k}A_{12}^{(k)}W_{k}-A_{21}^{(k)}-A_{22}^{(k)}W_{k}=0.
\label{invarK}%
\end{equation}
Choosing a subsequence if necessary, one can assume that the sequence
$\{W_{k}\}_{k=1}^{\infty}$ WOT-converges to some contraction $W\in
\mathcal{B}(H_{-},H_{+})$. It follows easily from the definition of
$M$-convergence that $W$ satisfies (\ref{invar}).
\end{proof}

\begin{proof}
[Deduction of Theorem \ref{PK} from Theorem \ref{stable}]Denote by
$(P_{-}^{(k)})_{k=1}^{\infty}$ and $(P_{+}^{(k)})_{k=1}^{\infty}$ increasing
sequences of finite-dimensional projections such that $P_{-}^{(k)}%
\overset{\text{sot}}{\rightarrow}P_{-}$ and $P_{+}^{(k)}\overset{\text{sot}%
}{\rightarrow}P_{+}$, and set $P^{(k)}=P_{-}^{(k)}+P_{+}^{(k)}$. Then the
operators $A^{(k)}=P^{(k)}AP^{(k)}$ are $J$-dissipative, finite-dimensional
and $M$-converge to $A$ (the condition $\Vert A_{12}^{(k)}-A_{12}%
\Vert\rightarrow0$ follows from the compactness of $A_{12}$). To see that each
$A^{(k)}$ has an MNPS it suffices to show that any $J$-dissipative operator in
a finite-dimensional indefinite metric space has an MNPS.
\end{proof}

The proof of the PK-Theorem in the finite-dimensional case was dropped in
\cite{Lom} as an easy one. In fact, the usual proof of this theorem for
matrices (see e.g. \cite{Gohb}) is not simple and is not direct: it goes via
study of $J$-expanding operators and application of Caley transform. To
present Lomonosov's result in the complete form we add a short direct
proof for the finite-dimensional case \textit{which again uses Theorem
\ref{stable}}.

\begin{proof}
[Completion of the proof of Theorem \ref{PK}]
Let $A$ be a $J$-dissipative operator on a finite-dimensional indefinite
metric space $H$. For each $t>0$, the operator $B=A+tJ$ satisfies the
condition of \textit{strong }$J$\textit{-dissipativity}:
\[
\mathrm{Im\,}[Bx,x]>0\text{ if }x\neq0.
\]
Since $A+tJ\rightarrow A$ when $t\rightarrow0$, Theorem \ref{stable} allows us
to assume that $A$ is strongly dissipative. In this case $A$ has no real
eigenvalues: if $Ax=tx$, for some $t\in\mathbb{R}$ and $0\neq x\in H$, then
$[Ax,x]=t[x,x]\in\mathbb{R}$, a contradiction. Let us denote by $H_{+}$ and
$H_{-}$ the spectral subspaces of $A$ corresponding to sets $\mathbb{C}%
_{+}=\{z\in\mathbb{C}:\mathrm{Im}$\thinspace$z>0\}$ and $\mathbb{C}_{-}%
=\{z\in\mathbb{C}:\mathrm{Im\,}z<0\}$, respectively. We will show that
subspaces $H_{+}$ and $H_{-}$ are positive and negative, respectively.

If an operator $T$ is strongly $J$-dissipative then also $-T^{-1}$ is strongly
$J$-dissipative. Indeed,
\[
-\operatorname{Im}\,[T^{-1}x,x]=\operatorname{Im}\,[x,T^{-1}%
x]=\operatorname{Im}\,[TT^{-1}x,T^{-1}x]>0
\]
if $x\neq0$. Since $A-t1$ is strongly $J$-dissipative, for each $t\in
\mathbb{R}$, we get that $-(A-t1)^{-1}$ is strongly $J$-dissipative. Now, for
each $0\neq x\in H_{+}$, one has
\[
x=\frac{i}{\pi}\int_{-\infty}^{\infty}\left(  A-t1\right)  ^{-1}xdt
\]
whence
\[
\lbrack x,x]=\operatorname{Re}\,[x,x]=-\operatorname{Im}\,\left(  \frac{1}%
{\pi}\int_{-\infty}^{\infty}[(A-t1)^{-1}x,x]dt\right)  >0.
\]
Thus $H_{+}$ is positive. Similarly, $H_{-}$ is negative.

So $H=H_{-}+H_{+}$ is the decomposition of $H$ into the direct sum of a
negative subspace and a positive subspace. It follows that $H_{\boldsymbol{-}%
}$ is an invariant MNPS. $\ \ \ $

\end{proof}

\bigskip

We add that

\begin{itemize}
\item In works of T. J. Azizov, H. Langer, A. A. Shkalikov and other
mathematicians Theorem \ref{invar} was extended to various classes of
unbounded operators (see for example \cite{Shkal} and references therein);

\item M. A. Naimark \cite{Naim1} proved that any commutative family $Q$ of
$J$-selfadjoint operators in a $\Pi_{k}$-space has a common invariant MNPS. It
follows that the result holds for any commutative family $Q$ of operators
which is $J$-\textit{symmetric}: $T\in Q$ implies $T^{\sharp}\in Q$.
\end{itemize}

\section{Fixed points}

Let us return to Krein's proof of the existence of invariant MNPS for
$J$-unitary operators. It is clear that any $J$-unitary operator $U$ maps any
MNPS onto an MNPS. Using the bijection $W\mapsto L_{W}$ between MNP subspaces
and contractions we see that $U$ determines the map $\phi_{U}$ from the closed
unit ball $\mathcal{B}_{1}(H_{-},H_{+})$ of the space $\mathcal{B}(H_{-}%
,H_{+})$ into itself. It is easy to obtain the direct expression of $\phi_{U}$
in terms of $U$:
\begin{equation}
\phi_{U}(W)=(U_{21}+U_{22}W)(U_{11}+U_{12}W)^{-1}\label{fr-li}%
\end{equation}
(we use notation from (\ref{matr})). It was shown in \cite{Kr} that if
$U_{12}$ is compact then the map $\phi_{U}$ is WOT-continuous; since
$\mathcal{B}_{1}(H_{-},H_{+})$ is WOT-compact, the fixed-point theorem implies
the existence of a contraction $W$ with $\phi_{U}(W)=W$. This means that
$L_{W}$ is invariant with respect to $U$. We get the following result:

\begin{theorem}
\cite{Kr} \label{MGK} Let $U$ be a $J$-unitary operator on a Krein space
$H=H_{+}+H_{-}$. If the \textquotedblleft corner \textquotedblright\ $U_{12}$
in the block-matrix of $U$ with respect to the decomposition $H=H_{+}+H_{-}$
is compact then $U$ has an invariant MNPS.
\end{theorem}

This result can be reformulated independently of the choice of the
decomposition $H=H_{+}+H_{-}$ and without matrix terminology:

\begin{theorem}
If $J$-unitary operator $U$ on a Krein space $H$ is a compact perturbation of
an operator that preserves a maximal negative subspace, then it has an
invariant MNPS.
\end{theorem}

To prove this
let $U=R+K$, where $K$ is compact, $R$ preserves a maximal negative subspace
$L\subset H$. Let $M=L^{\bot}$, and let $P$ be the projection onto $L$ along
$M$. Then
\[
(1-P)UP=(1-P)RP+(1-P)KP=(1-P)KP
\]
is a compact operator. But $(1-P)UP$ is the corner of the block-matrix $U$
with respect to the decomposition $H=L+M$. So, by Krein's theorem, $U$ has an
invariant MNPS.

\bigskip

Note that for $\Pi_{k}$-spaces the assumption of compactness of $U_{21}$ is
automatically satisfied, so Krein's Theorem implies that any $J$-unitary
operator on a $\Pi_{k}$-space has an invariant MNPS.

\bigskip

The fractional-linear maps $\phi_{U}$ defined by (\ref{fr-li}) preserve the
open unit ball $\mathfrak{B}=\{X\in\mathcal{B}(H_{-},H_{+}):\Vert X\Vert<1\}$
and their restrictions to $\mathfrak{B}$ form the group of all biholomorphic
automorphisms of $\mathfrak{B}$ (we refer to \cite{AI} or \cite{KSh} for more
information). So the existence of fixed points for such maps and families of
such maps are of independent interest. After Naimark's result it was natural
to try to prove the existence of common fixed points for commutative sets of
fractional-linear maps. Note that this does not follow directly from Naimark's
Theorem, because the maps $\phi_{U}$ and $\phi_{V}$ commute if and only if the
operators $U$ and $V$ commute \textit{up to a scalar multiple}: $UV=\lambda
VU$, $\lambda\in\mathbb{C}$. The positive answer was obtained by J.W. Helton:

\begin{theorem}
\cite{Helt1}\label{Helton} Let $H_{1}$, $H_{2}$ be Hilbert spaces and $\dim
H_{1}<\infty$. Then any commutative family of fractional-linear maps of the
closed unit ball in $\mathcal{B}(H_{1},H_{2})$ has a common fixed point.
\end{theorem}

This result implies Naimark's Theorem, but the proof uses it. Another result
of Helton \cite{Helt2} based on the consideration of fractional-linear maps
states that a commutative group of $J$-unitary operators on a Krein space
$H_{1}\oplus H_{2}$ has an invariant maximal positive subspace if it contains
a compact perturbation of an operator $A\oplus B$ with $\sigma(A)\cap
\sigma(B)=\emptyset$. This extends the Naimark Theorem because the identity
operator $1$ in a $\Pi_{k}$-space is a compact perturbation of $J$.

The following result on fixed points of groups of fractional-linear maps was
proved by M. Ostrovskii, V. S. Shulman and L. Turowska \cite{OShT, OShT2} (see
also \cite{Sh0} where the case $k=1$ was considered).

\begin{theorem}
\label{orbit} Let $\dim H_{2}=k<\infty$ and let a group $\Gamma$ of
fractional-linear maps of the open unit ball $\mathfrak{B}$ in $\mathcal{B}%
(H_{2},H_{1})$ have an orbit separated from the boundary \emph{(}$\sup
_{\phi\in\Gamma}\Vert\phi(K)\Vert<1$, for some $K\in\mathfrak{B}$\emph{)}.
Then there is $K_{0}\in\mathfrak{B}$ such that $\phi(K_{0})=K_{0}$, for all
$\phi\in\Gamma$.
\end{theorem}

\begin{corollary}
\label{bounded} Any bounded group of $J$-unitary operators in a $\Pi_{k}%
$-space has an invariant dual pair of subspaces.
\end{corollary}

We will obtain some related results in the next two sections.

\section{Orthogonalization of bounded representations}

In many situations (see the book \cite{Pisier} for examples and discussions)
it is important to know if a given representation $\pi$ of a group $G$ in a
Hilbert space is similar to a unitary representation:
\[
\pi(g)=V^{-1}U(g)V,\text{ for all }g\in G,
\]
the operators $U(g)$ are unitary, $V$ is an invertible operator. The infimum
$c(\pi)$ of values $\Vert V\Vert\Vert V^{-1}\Vert$ for all possible $V$'s, is
called the constant of similarity of $\pi$. It is obvious that a
representation can be similar to a unitary one only if it is bounded:
\[
\Vert\pi\Vert:=\sup_{g\in G}\Vert\pi(g)\Vert<\infty;
\]
clearly $\Vert\pi\Vert\leq c(\pi)$.

By a quadratic form we mean a function $\Phi(x)=(Ax,x)$ on a Hilbert space
$H$, where $A$ is an invertible selfadjoint operator on $H$. Changing the
scalar product if necessary, one can reduce the situation to the case that
\begin{equation}
\Phi(x)=(P_{1}x,x)-(P_{2}x,x), \label{form}%
\end{equation}
where $P_{1}$ and $P_{2}$ are projections with $P_{1}+P_{2}=1$ (if a form is
given as above then $P_{1}$ and $P_{2}$ are spectral projections of $A$
corresponding to the intervals $(-\infty,0)$ and $(0,\infty)$. So we consider only
forms given by (\ref{form}). The number $\dim(P_{2}H)$ is called the number of
negative squares of $\Phi$.

A representation $\pi$ is said to preserve the form (\ref{form}) if $\Phi
(\pi(g)x) = \Phi(x)$, for all $x\in H, g\in G$.

\begin{theorem}
\label{ort} Any bounded representation $\pi$ preserving a form with finite
number of negative squares is similar to a unitary representation. Moreover,
\begin{equation}
c(\pi)\leq2\Vert\pi\Vert^{2}+1. \label{cepi}%
\end{equation}

\end{theorem}

The first statement of the theorem was proved in \cite{OShT}; to prove the
inequality (\ref{cepi}) we will repeat some steps of the proof in \cite{OShT}
adding necessary changes and estimations.

We begin with a general result on fixed points of groups of isometries.

Let us say that a metric space $(\mathcal{X},d)$ is \textit{ball-compact} if a
family of balls
\[
E_{a,r}=\{x\in X:d(a,x)\leq r\}
\]
has non-void intersection provided each its finite subfamily has non-void
intersection (see \cite{Tak1}).

A subset $M\subset\mathcal{X}$ is called \textit{ball-convex} if it is the
intersection of a family of balls. The compactness property extends from balls
to ball-convex sets: if $(\mathcal{X},d)$ is ball-compact, then a family
$\{M_{\lambda}:\lambda\in\Lambda\}$ of ball-convex subsets of $\mathcal{X}$
has non-void intersection if each its finite subfamily has non-void intersection.

The \textit{diameter} of a subset $M\subset\mathcal{X}$ is defined by
\begin{equation}
\text{diam}(M)=\sup\{d(x,y):x,y\in M\}.\label{diam}%
\end{equation}
A point $a\in M$ is called \textit{diametral} if
\[
\sup\{d(a,x):x\in M\}=\mathrm{diam}(M).
\]
A metric space $\mathcal{X}$ is said to have \textit{normal structure} if
every ball-convex subset of $\mathcal{X}$ with more than one element has a
non-diametral point.

\begin{lemma}
\label{fp-gen} Suppose that a metric space $(\mathcal{X},d)$ is ball-compact
and has normal structure. If a group $\Gamma$ of isometries of $(\mathcal{X}%
,d)$ has a bounded orbit $O$, then it has a fixed point $x_{0}$.
Moreover\emph{,} $x_{0}$ belongs to the intersection of all ball-convex
subsets containing $O$.
\end{lemma}

\begin{proof}
The family $\Phi$ of all balls containing $O$ is non-void. Since $O$ is
invariant under $\Gamma$, the family $\Phi$ is also invariant: $g(E)\in\Phi$,
for each $E\in\Phi$. Hence the intersection $M_{1}$ of all elements of $\Phi$
is a non-void $\Gamma$-invariant ball-convex set; moreover, it follows easily
from the definition that $M_{1}$ is the intersection of all ball-convex
subsets containing $O$.

Thus the family $\mathcal{M}$ of all non-void $\Gamma$-invariant ball-convex
subsets of $M_{1}$ is non-void. Therefore the intersection of a decreasing
chain of sets in $\mathcal{M}$ belongs to $\mathcal{M}$ and, by Zorn Lemma,
$\mathcal{M}$ has minimal elements. Our aim is to prove that any minimal
element $M$ of $\mathcal{M}$ consists of one point.

Assuming the contrary, let $\text{diam}(M)=\alpha>0$. Since $(\mathcal{X},d)$
has normal structure, $M$ contains a non-diametral point $a$. It follows that
$M\subset\{x\in\mathcal{X}:~d(a,x)\leq\delta\}$ for some $\delta<\alpha$. Set
\[
D=\bigcap_{b\in M}E_{b,\delta}.
\]
The set $D$ is non-void because $a\in D$. Furthermore, $D$ is ball-convex by
definition. To see that $D$ is a proper subset of $M$ take $b,c\in M$ with
$d(b,c)>\delta$, then $c\notin E_{b,\delta}$, hence $c\notin D$.

Since $\Gamma$ is a group of isometric transformations and $M$ is invariant
under each element of $\Gamma$, $D$ is $\Gamma$-invariant. We get a
contradiction with the minimality of $M$.

Thus $M=\{x_{0}\}$, for some $x_{0}\in M_{1}$.
\end{proof}

Let now $H_{1}$ and $H_{2}$ be Hilbert spaces, $\dim H_{2}<\infty$. We denote
by $\mathfrak{B}$ the open unit ball of the space $\mathcal{B}(H_{2},H_{1})$
of all linear operators from $H_{2}$ to $H_{1}$.

\medskip

For each $A\in\mathfrak{B}$, we define a transformation $\mu_{A}$ of
$\mathfrak{B}$ (\textit{a M\"{o}bius transformation}) by setting
\begin{equation}
\mu_{A}(X)=(1-AA^{\ast})^{-1/2}(A+X)(1+A^{\ast}X)^{-1}(1-A^{\ast}A)^{1/2}.
\label{mobius}%
\end{equation}
It can be easily checked that $\mu_{A}(0)=A$ and $\mu_{A}^{-1}=\mu_{-A}$, for
each $A\in\mathfrak{B}$.

We set
\begin{equation}
\rho(A,B)=\tanh^{-1}(||\mu_{-A}(B)||). \label{dist0}%
\end{equation}

It was proved in \cite[Theorem 6.1]{OShT} that the space $(\mathfrak{B},\rho)$
is ball-compact and has a normal structure. It can be also verified that
$\rho$ coincides with the Carath\'{e}odory distance $c_{\mathfrak{B}}$ in
$\mathfrak{B}$. Therefore all biholomorphic maps of $\mathfrak{B}$ preserve
$\rho$. Applying Lemma \ref{fp-gen} we get the following statement.

\begin{lemma}
\label{fp} If a group of biholomorphic transformations of $\mathfrak{B}$ has
an orbit contained in the ball $r\overline{\mathfrak{B}}=\{X\in\mathcal{B}%
(H_{2},H_{1}):\Vert X\Vert\leq r\}$, where $r<1$, then it has a fixed point
$K\in r\overline{\mathfrak{B}}$.
\end{lemma}

As we know, biholomorphic transformations of $\mathfrak{B}$ are just
fractional-linear transformations corresponding to $J$-unitary operators in
$H=H_{1}+H_{2}$ with the indefinite scalar product $[x,y]=(P_{1}%
x,y)-(P_{2}x,y)$.

Let us denote by $\mathcal{T}$ the group of all fractional-linear
transformations of $\mathfrak{B}$. Note that $\mathcal{T}$ contains all
M\"{o}bius maps. Indeed it can be easily checked that $\mu_{A}=\phi_{M_{A}}$
where $M_{A}$ is the $J$-unitary operator with the matrix%
\[%
\begin{pmatrix}
(1_{H}-A^{\ast}A)^{-1/2} & A^{\ast}(1_{K}-AA^{\ast})^{-1/2}\\
A(1_{H}-A^{\ast}A)^{-1/2} & (1_{K}-AA^{\ast})^{-1/2}%
\end{pmatrix}
\]
Since $\mu_{A}(0) = A$ we see that $\mathcal{T}$ acts transitively on
$\mathfrak{B}$.

\begin{lemma}
\label{estimates} Let $U$ be a $J$-unitary operator on a $\Pi_{k}$-space $H$,
$\phi_{U}$ the corresponding fractional-linear map and $A = \phi_{U}(0)$. Let
$C = \|U\|$ and $r = \|A\|$. Then
\begin{equation}
C\leq\sqrt{(1+r)(1-r)^{-1}}. \label{direct}%
\end{equation}
and
\begin{equation}
r\leq\sqrt{(C^{2}-1)/(C^{2}+1)}. \label{codirect}%
\end{equation}

\end{lemma}

\begin{proof}
Let $V = M_{A}^{-1}U$, then $\phi_{V}(0) = (\mu_{A})^{-1}(A) = 0$ so the $J$-unitary operator $V$ preserves
subspaces $H_{1}$ and $H_{2}$; it follows that $V$ is a unitary operator on $H$.
Thus $\|U\| = \|M_{A}V\| = \|M_{A}\|$, so it suffices to prove the inequalities (\ref{direct}) and
(\ref{codirect}) for $U = M_A$.

Let, for brevity, $S=(1+A^{\ast}A)(1-A^{\ast}A)^{-1}$ and $T=(1+AA^{\ast
})(1-AA^{\ast})^{-1}$. For any $z=x_{1}+x_{2}\in H_{1}+H_{2}$, a direct
calculation gives
\[
\Vert M_{A}z\Vert^{2}=(Sx_{1},x_{1})+(Tx_{2},x_{2})+4\mathrm{Re\,}%
(((1-AA^{\ast})^{-1}Ax_{1},x_{2}).
\]
Recall that in our notations $\Vert A\Vert=r$, $\Vert M_{A}\Vert=C$. Since
\[
\Vert S\Vert=\Vert T\Vert=(1+r^{2})(1-r^{2})^{-1}%
\]
and
\begin{align*}
\Vert(1-AA^{\ast})^{-1}A\Vert &  =\Vert(1-AA^{\ast})^{-1}AA^{\ast}(1-AA^{\ast}%
)^{-1}\Vert^{1/2}\\
&  =r(1-r^{2})^{-1},
\end{align*}
we get
\begin{align*}
\Vert M_{A}z\Vert^{2}  &  \leq(1+r^{2})(1-r^{2})^{-1}(\Vert x_{1}\Vert
^{2}+\Vert x_{2}\Vert^{2})+4r(1-r^{2})^{-1}\Vert x_{1}\Vert\Vert x_{2}\Vert\\
&  \leq(1+r^{2})(1-r^{2})^{-1}(\Vert x_{1}\Vert^{2}+\Vert x_{2}\Vert
^{2})+2r(1-r^{2})^{-1}(\Vert x_{1}\Vert^{2}+\Vert x_{2}\Vert^{2})\\
&  =(1+r)(1-r)^{-1}\Vert z\Vert^{2},
\end{align*}
which proves (\ref{direct}).

On the other hand, for $x\in H_{2}$, we have
\begin{align*}
\Vert M_{A}x\Vert^{2}  &  =(AA^{\ast}(1-AA^{\ast})^{-1}x,x)+((1-AA^{\ast
})^{-1}x,x)\\
&  =\left\Vert \sqrt{T}x\right\Vert ^{2}%
\end{align*}
whence
\[
\sqrt{(1+r^{2})/(1-r^{2})}=\left\Vert \sqrt{T}\right\Vert \leq\Vert M_{A}%
\Vert=C.
\]
This shows that the inequality (\ref{codirect}) holds.

\end{proof}

\begin{proof}
[The proof of (\ref{cepi}) in Theorem \ref{ort}] Now recall that by the
assumptions of theorem we have a bounded group $\{\pi(g):g\in G\}$ of
operators on a Hilbert space $H$ preserving the form $\Phi$ given by
(\ref{form}). Introducing the indefinite scalar product $[x,y]=(P_{1}%
x,y)-(P_{2}x,y)$ on $H,$ we convert $H$ into a $\Pi_{k}$-space:
\[
H=H_{1}+H_{2},\text{ where }H_{i}=P_{i}H.
\]
Since $\Phi(x)=[x,x],$ all operators $\pi(g)$ are $J$-unitary. Let
$\Gamma=\{\phi_{\pi(g)}:g\in G\}$, the corresponding group of
fractional-linear transformations of the open unit ball $\mathfrak{B}$ of
$\mathcal{B}(H_{2},H_{1})$, and consider the $\Gamma$-orbit $O$ of the
point $0\in\mathfrak{B}$.

For $g\in G$, the inequality $\|\pi(g)\|\le \|\pi\|$, Lemma \ref{estimates} and monotonicity of the function $t\mapsto \sqrt{(t^{2}-1)/(t^{2}+1)}$ imply that
 $$\Vert\phi_{\pi(g)}(0)\Vert\leq
R:=\sqrt{(\|\pi\|^{2}-1)/(\|\pi\|^{2}+1)},$$
so $O\subset R\overline{\mathfrak{B}}$.
By Lemma \ref{fp}, there is an operator $K\in R\overline{\mathfrak{B}}$ such
that $\phi_{\pi(g)}(K)=K$, for all $g\in G$.

Let $V=M_{K}$ and $U(g)=V\pi(g)V^{-1}$ for each $g\in G$. Then $U(g)$ is
$J$-unitary and
\[
\phi_{U(g)}(0)=\mu_{K}\circ\phi_{\pi(g)}\mu_{-K}(0)=\mu_{K}\left(  \phi
_{\pi(g)}(K)\right)  =\mu_{K}(K)=0.
\]
Therefore $U(g)$ preserves $H_{1}$ and $H_{2}$. Since $(x,y)=[x_{1}%
,y_{1}]-[x_{2},y_{2}]$, where $x_{i}=P_{i}x\in H_{i}$, $y_{i}=P_{i}y\in H_{i}%
$, $i=1,2$, we see that
\begin{align*}
(U(g)x,U(g)y) &  =[U(g)x_{1},U(g)y_{1}]-[U(g)x_{2},U(g)y_{2}]=[x_{1}%
,y_{1}]-[x_{2},y_{2}]\\
&  =(x,y),
\end{align*}
for all $x,y\in H$. Thus $U(g)$ is a unitary operator in $H$. We proved that
$\pi$ is similar to a unitary representation\textbf{;} moreover, by Lemma
\ref{estimates},
\begin{align*}
c(\pi) &  \leq\Vert V\Vert\Vert V^{-1}\Vert=\Vert M_{K}\Vert\Vert M_{-K}%
\Vert\leq\sqrt{(1+R)(1-R)^{-1}}^{2}\\
&  =(1+R)(1-R)^{-1}.
\end{align*}
Since $R=\sqrt{(\Vert\pi\Vert^{2}-1)/(\Vert\pi\Vert^{2}+1)}$, we get that
\[
c(\pi)\leq\Vert\pi\Vert^{2}+1+\sqrt{\Vert\pi\Vert^{4}-1}<2\Vert\pi\Vert^{2}+1,
\]
which completes the proof.
\end{proof}

The fact that our estimate of the similarity degree does not depend on the
number of negative squares leads to the conjecture that the result extends to
representations preserving forms with infinite number of negative squares. We
shall see now that this is not true.

It is known (see \cite{Pisier}) that for some groups there exist bounded
representations which are not similar to unitary ones (there is a conjecture
that all non-amenable groups have such representations). Let $\pi$ be such a
representation of a group $G$ on a Hilbert space $H$. We define a
representation $\tau$ of $G$ on $\mathcal{H}=H\oplus H$ by setting%
\[
\tau(g)=%
\begin{pmatrix}
\pi(g) & 0\\
0 & \pi(g^{-1})^{\ast}%
\end{pmatrix}
.
\]
Clearly $\tau$ is bounded. Moreover, it is not similar to a unitary
representation because otherwise $\pi$, being its restriction to an invariant
subspace, would be similar to a restriction of a unitary representation, which
is again unitary.

The space $\mathcal{H}$ is a Krein space with respect to the inner product
$[x_{1}\oplus y_{1},x_{2}\oplus y_{2}]=(x_{1},y_{2})+(y_{1},x_{2})$. Indeed,
$\mathcal{H}=\mathcal{H}_{+}+\mathcal{H}_{-}$, where the subspaces $\mathcal{H}_{+}=\{x\oplus
x:x\in H\}$ and $\mathcal{H}_{-}=\{x\oplus(-x):x\in H\}$ are respectively
positive and negative. It remains to check that the form $\Phi(x\oplus
y)=[x\oplus y,x\oplus y]$ is preserved by operators $\tau(g)$:
\begin{align*}
\lbrack\tau(g)(x\oplus y),\tau(g)(x\oplus y)]  &  =(\pi(g)x,\pi(g^{-1})^{\ast
}y)+(\pi(g^{-1})^{\ast}y,\pi(g)x)\\
&  =(x,y)+(y,x)=[x\oplus y,x\oplus y].
\end{align*}

\section{Quasi-positive definite functions}

Recall that a function $\phi$ on a group $G$ is \textit{positive definite}
(PD, for brevity) if $\phi(g^{-1})=\overline{\phi(g)}$, for $g\in G$, and the
matrices $A_{n}=(\phi(g_{i}^{-1}g_{j}))_{i,j=1}^{n})$ have no negative
eigenvalues, for all $n\in\mathbb{N}$ and all $n$-tuples $g_{1},...,g_{n}\in
G$. In other words, the quadratic forms $\sum_{i,j=1}^{n}\phi(g_{i}^{-1}%
g_{j})z_{i}\overline{z_{j}}$ are positive for all $n\in\mathbb{N}$ . A famous
theorem of Bochner \cite{Boch} states that all such functions can be described as
matrix elements of unitary representations:
\[
\phi(g)=(\pi(g)x,x),
\]
where $\pi$ is a unitary representation of $G$ in a Hilbert space $H$ and
$x\in H$.

We say that $\phi$ is PD \textit{of finite type} if the corresponding
representation is finite-dim\-en\-si\-onal.
It could be proved that $\phi$ is PD of finite type if and only if it
satisfies the condition
\[
\phi(g^{-1}h)=\sum_{i=1}^{m}a_{i}(g)\overline{a_{i}(h)}\text{ for all }g,h\in
G,
\]
where $a_{i}$ are some functions on $G$. For example, the function $\cos x$ is
PD of finite type on $\mathbb{R}$.

A function $\phi$ on a group $G$ is called \textit{quasi-positive definite}
(QPD hereafter) if $\phi(g^{-1})=\overline{\phi(g)}$, for $g\in G$, and there
is $k\in\mathbb{N}$ such that, for any $n\in\mathbb{N}$ and any $n$-tuple
$g_{1},...,g_{n}\in G$, the matrix $(\phi(g_{i}^{-1}g_{j}))_{i,j=1}^{n})$ has
at most $k$ negative eigenvalues. In other words, the quadratic form
$\sum_{i,j=1}^{n}\phi(g_{i}^{-1}g_{j})z_{i}\overline{z_{j}}$ should have at
most $k$ negative squares.

The study of QPD functions was initiated by M. G. Krein \cite{Kr2} motivated
by applications to probability theory --- in particular, to infinite divisible
distributions and, more generally, to stochastic processes with stationary
increments. Other applications of theory of QPD functions are related to
moment problems, Toeplitz forms and other topics of functional analysis, see
\cite{Sasv, Sakai} and references therein.

It is easy to see that the difference $a(g)-b(g)$ of two PD functions is a QPD
function if $b$ is of finite type. Clearly such QPD functions are bounded. The
following theorem shows that all bounded QPD functions are of this type.

\begin{theorem}
\label{funct} Every bounded QPD function $\phi$ can be written in the form
\[
\phi(g)=\phi_{1}(g)-\phi_{2}(g),
\]
where $\phi_{1}$ is a PD function and $\phi_{2}$ is a PD function of finite type.
\end{theorem}

\begin{proof}
There is a standard way to associate with $\phi$ a $J$-unitary representation
of $G$ on a $\Pi_{k}$-space. Let $W$ be the linear space of all finitely
supported functions on $G$; we define an indefinite scalar product
$[\cdot,\cdot]$ on $W$ by setting
\begin{equation}
\lbrack f_{1},f_{2}]=\sum_{g,h\in G}f_{1}(g)\overline{f_{2}(h)}\phi
(g^{-1}h).\label{prod}%
\end{equation}
For each $g\in G$ we define an operator $T_{g}$ on $W$ by setting
$T_{g}f(h)=f(g^{-1}h)$. It is easy to check that the operators $T_{g}$
preserve $[\cdot,\cdot]$, that is, $[T_{g}f_{1},T_{g}f_{2}]=[f_{1},f_{2}]$,
for all $f_{1},f_{2}$. Clearly, the map $g\mapsto T_{g}$ is a representation
of $G$ on $W$.

Defining by $\varepsilon_{g}$, for $g\in G$, the function on $G$ equal 1 at
$g$ and 0 at other elements, we see that the matrix $(\phi(g_{i}^{-1}%
g_{j}))_{i,j=1}^{n}$ is the Gram matrix for the family $\varepsilon_{g_{1}%
},...,\varepsilon_{g_{n}}$. Since the linear span of vectors $\varepsilon_{g}$
coincides with $W$, the condition "$\phi$ is QPD\;"\; implies that the dimension of any
negative subspace of $W$ does not exceed $k$. It follows that $W=W_{1}+H_{-},$
where $W_{1}$ is a positive subspace, $H_{-}$ is negative and $\dim H_{-}=k$.
Denoting by $H_{+}$ the completion of $W_{1}$ with respect to the scalar
product $[\cdot,\cdot]|_{W_{1}}$, we get a $\Pi_{k}$-space $H=H_{+}+H_{-}$. It
is not difficult to show that operators $T_{g}$ extend to bounded $J$-unitary
operators $U(g)$ on $H$. It follows easily from the definition that
$\phi(g)=[U(g)f,f]$, where $f$ is the image of $\varepsilon_{e}$ in $H$.

Since $\phi$ is bounded, the representation $U$ is bounded (see for example
\cite[Theorem 3.2]{Sakai}). By Corollary \ref{bounded}, there is a decomposition
$H=K_{+}+K_{-}$ where $K_{+}$ is positive, $K_{-}$ is negative, and both
subspaces are invariant for operators $U(g)$. In other words, the operators $U(g)$
commute with the projection $P$ on $K_{+}$. Setting $f_{+}=Pf$,
$f_{-}=(1-P)f$, we get
\begin{align*}
\phi(g) &  =[U(g)f,f]=[U(g)f_{+},f_{+}]+[U(g)f_{-},f_{-}]=(U(g)f_{+}%
,f_{+})-(U(g)f_{-},f_{-})\\
&  =\phi_{1}(g)-\phi_{2}(g),
\end{align*}
which is what we need because the functions $\phi_{1}$ and $\phi_{2}$ are PD,
and $\phi_{2}$ is of finite type.
\end{proof}

For amenable groups the result was proved by K. Sakai \cite{Sakai}.

\section{$J$-symmetric algebras and Burnside-type theorems}

As in linear algebra, after proving the existence of a nontrivial invariant
subspace (IS, for brevity) for a single operator, one looks for conditions
under which a family of operators has a common IS. Since the lattice
$\mathrm{Lat}(E)$ of invariant subspaces of a family $E\subset\mathcal{B}(H)$
coincides with $\mathrm{Lat}(\mathcal{A}(E))$, where $\mathcal{A}(E)$ is the
algebra generated by $E$, it is reasonable to restrict ourself by study of
non-positive invariant subspaces for algebras (more precisely, for
$J$-symmetric operator algebras in a $\Pi_{k}$-space $H$). Thus one may
rewrite the Naimark's Theorem in the form: all commutative $J$-symmetric
algebras in $H$ have invariant MNPS. What else?

For algebras of operators in a finite-dimensional space, the problem of
existence of invariant subspaces was completely solved by W. Burnside
\cite{Burn}: the only algebra that has no IS is the algebra of all operators.
For infinite-dimensional Hilbert spaces, the problem is unsolved: it is
unknown if there exists an algebra $A\subset\mathcal{B}(H)$ which has no
(closed) IS and is not WOT-dense in $\mathcal{B}(H)$. In presence of compact
operators the answer was given by Victor Lomonosov \cite{Lom1}: if an algebra
$A$ contains at least one non-zero compact operator, then either $A$ has an
invariant subspace or it is WOT-dense in $\mathcal{B}(H)$. In fact, he proved
much more: if an algebra $A$ contains a non-zero compact operator and has no
invariant subspaces then the \textit{norm}-closure of $A$ contains the algebra
$\mathcal{K}(H)$ of all compact operators. These results were further extended
in Lomonosov's work \cite{Lom2}.

For *-algebras of operators, von Neumann's \textit{double commutant Theorem}
\cite{Dix} immediately implies a Burnside-type result: a *-algebra of
operators has an invariant subspace if and only if it is not WOT-dense in
$\mathcal{B}(H)$.

Since Theorem \ref{PK} establishes that a $J$-symmetric operator has an
invariant subspace, it leads to the traditional Burnside-type problem for
$J$-symmetric algebras: which $J$-symmetric algebras of operators on a space
of $\Pi_{k}$-type have no invariant subspaces?

The first answer was given by R. S. Ismagilov \cite{Ism}: a $J$-symmetric
WOT-closed algebra $A$ in a $\Pi_{k}$-space either has an invariant subspace
or coincides with $\mathcal{B}(H)$ (this work presents also another proof of
Pontryagin's Theorem, which is short but based on a deep result of J. Schwartz
\cite{Schw} about invariant subspaces of finite-rank perturbations of
selfadjoint operators). Furthermore A. I. Loginov and V. S. Shulman \cite{LSh}
(see \cite{KLSh} and \cite{KSh} for a more transparent presentation) proved
the corresponding result for \textit{norm-closed} $J$-symmetric algebras in
$\Pi_{k}$-spaces was obtained: a $J$-symmetric algebra $A\subset
\mathcal{B}(H)$ has no invariant subspaces if and only if its norm-closure
contains the algebra $\mathcal{K}(H)$. The proof is quite complicated and uses
the striking theorem of J. Cuntz \cite{Cu} about $C^{\ast}$-equivalent Banach *-algebras.

The following Burnside-type result is more closely related to the
Pontryagin-Krein Theorem: it describes $J$-symmetric algebras that have no
\textit{non-positive} invariant subspaces. To formulate it let us consider a
Hilbert space $E$ and the direct sum $H=\oplus_{i=1}^{n}E_{i}$ of $n\leq
\infty$ copies of $E$. Let $\mathcal{B}(E)^{(n)}$ be the algebra of all
operators on $H$ of the form $T\oplus T\oplus...$, where $T\in\mathcal{B}(E)$.
On each summand $E_{i}=E$ in $H$ we choose a projection $P_{i}$ with
$0\leq\dim P_{i}E=k_{i}<\dim E$, assuming that $\sum_{i}k_{i}=k<\infty$, and
set $P=P_{1}\oplus P_{2}\oplus...$, $J=1-2P$. Then $H$ is a $\Pi_{k}$-space
with respect to the inner product $[x,y]=(Jx,y)$. The algebra $\mathcal{B}%
(E)^{(n)}$ is clearly $J$-symmetric; $J$-symmetric algebras of this form are
called \textit{model algebras}.

\begin{theorem}
\label{bur} A WOT closed $J$-symmetric algebra $A$ on a $\Pi_{k}$-space $H$
does not have non-positive invariant subspaces if and only if it is a direct
$J$-orthogonal sum of a $W^{\ast}$-algebra on a Hilbert space and a finite
number of model algebras.
\end{theorem}

The proof can be easily deduced from \cite[Theorem 13.7]{KSh} that gives a
description of all algebras that have no neutral invariant subspaces. To
describe \textit{norm-closed} $J$-symmetric algebras without non-positive
invariant subspaces one should replace in Theorem \ref{bur} a W*-algebra by a
C*-algebra and model algebras $\mathcal{B}(E)^{(n)}$ by the algebras $A^{(n)}%
$, where $A\subset\mathcal{B}(E)$ is a C*-algebra containing $\mathcal{K}%
(E)$.\bigskip

Another natural version of the problem is to describe Banach *-algebras with
the property that all their $J$-symmetric representations in a $\Pi_{k}$-space
have MNPS. It is shown in \cite[Theorem 19.4]{KSh} that this property is
equivalent to the absence of irreducible $\Pi_{k}$-representations; let us
denote by ($\mathcal{K}$) the class of all Banach *-algebras that possess it.

It follows from Naimark's Theorem that ($\mathcal{K}$) contains all
commutative algebras. On the other hand Theorem \ref{ort} implies that any
Banach algebra, generated by a bounded subgroup of unitary elements belongs to
($\mathcal{K}$). This implies that ($\mathcal{K}$) contains all C*-algebras
(this was proved earlier in \cite{Sh1}).

Recall that a Banach *-algebra $A$ is \textit{Hermitian} if all its
selfadjoint elements have real spectra. Let us say that $A$ is \textit{almost
Hermitian} if the elements with real spectra are dense in the space of all
selfadjoint elements. It is proved in \cite[Corollary 20.6]{KSh} that all
almost Hermitian algebras belong to ($\mathcal{K}$); this result has
applications to the study of unbounded derivations of $C^{\ast}$-algebras (see
\cite{KSh}).

It is known that the group algebras $L^{1}(G)$ of locally compact groups are
not Hermitian for some $G$ (the Referee kindly informed us about a recent result of Samei and Wiersma \cite{SW} which states that $L^{1}(G)$ is not Hermitian if $G$ is not amenable). It is not known if all algebras $L^{1}(G)$ are almost
Hermitian. Nevertheless all $L^{1}(G)$ belong to ($\mathcal{K}$); moreover, the
following result holds.

\begin{theorem}
\label{grA} If $G$ is a locally compact group then any $J$-symmetric
representation of $L^{1}(G)$ on a $\Pi_{k}$-space $H$ has invariant dual pair
of subspaces.
\end{theorem}

We begin the proof of this theorem with a general statement which is
undoubtedly known but it is difficult to give a precise reference.

Recall that the {\it essential subspace} for a representation $D$ of an algebra $A$ on a Banach space $X$ is the closure of the linear span $D(A)X$ of all vectors $D(a)x$, where $a\in A$, $x\in X$. If the  essential subspace for $D$ coincides with $X$ then $D$ is
called \textit{essential}.

\begin{lemma}
\label{ext} Let $L$ be an ideal of a Banach algebra $A$, and $D:L\rightarrow
\mathcal{B}(X)$ be a bounded essential representation of $L$ in a Banach space
$X$. If $L$ has a bounded approximate identity $\{u_{n}\},$ then $D$ extends
to a bounded representation $\widetilde{D}$ of $A$ in $X$, and $\left\Vert
\widetilde{D}\right\Vert \leq C\left\Vert D\right\Vert $ where $C=\sup
_{n}\Vert u_{n}\Vert$.
\end{lemma}

\begin{proof}
Let us show that
\[
\left\Vert \sum_{i=1}^{n}D(ab_{i})x_{i}\right\Vert \leq C\Vert D\Vert\Vert
a\Vert\left\Vert \sum_{i=1}^{n}D(b_{i})x_{i}\right\Vert ,
\]
for any $a\in A,b_{i}\in L,x_{i}\in X$. Indeed,
\begin{align*}
\left\Vert \sum_{i=1}^{n}D(au_{n}b_{i})x_{i}\right\Vert  &  =\left\Vert
\sum_{i=1}^{n}D((au_{n})b_{i})x_{i}\right\Vert \\
&  =\left\Vert D(au_{n})\left(  \sum_{i=1}^{n}D(b_{i})x_{i}\right)
\right\Vert \\
&  \leq\left\Vert D(au_{n})\right\Vert \left\Vert \sum_{i=1}^{n}D(b_{i}%
)x_{i}\right\Vert \\
&  \leq C\left\Vert D\right\Vert \left\Vert a\right\Vert \left\Vert \sum
_{i=1}^{n}D(b_{i})x_{i}\right\Vert ,
\end{align*}
and it remains to note that
\[
\left\Vert \sum_{i=1}^{n}D(ab_{i})x_{i}-\sum_{i=1}^{n}D(au_{n}b_{i}%
)x_{i}\right\Vert \rightarrow0\text{ when }n\rightarrow\infty.
\]

Now we may define a map $T_{a}$ on the space $D(L)X$ by setting
\[
T_{a}\left(  \sum_{i=1}^{n}D(b_{i})x_{i}\right)  =\sum_{i=1}^{n}D(ab_{i}%
)x_{i}\text{ for all }b_{i}\in L\text{ and }x_{i}\in X.
\]
By the above,  $T_{a}$ is a well defined linear operator on
$D(L)X$ and
\[
\Vert T_{a}\Vert\leq C\Vert D\Vert\Vert a\Vert.
\]
Denoting by $\widetilde{D}(a)$ the closure of $T_{a}$, we obtain an operator
on $X$ with
\[
\Vert\widetilde{D}(a)\Vert\leq\Vert C\Vert D\Vert\Vert a\Vert.
\]
It is easy to see that the map $\widetilde{D}:a\mapsto\widetilde{D}(a)$ is a
representation of $A$ on $X$, extending $D$.
\end{proof}

Now we need a result about $J$-symmetric representations of *-algebras. Recall that a closed subspace $L$ of an indefinite metric space $H$ is {\it non-degenerate} if $L\cap L^{\bot} = 0$.

\begin{lemma}
\label{decomp} Let a $^{\ast}$-algebra $L$ have a bounded approximate identity
$\{u_{n}\}$, and let $D$ be a $J$-symmetric representation of $L$ on a Krein
space $H$. Then the essential subspace $H_0 = \overline{D(L)H}$ of  $D$ is
non-degenerate, and $H_{0}^{\bot}\subset\ker{D(L)}$.
\end{lemma}

\begin{proof}
Let $K=H_{0}\cap H_{0}^{\bot}$. For any $x\in H$, $y\in H_{0}^{\bot}$ and
$a\in L$ we have $[x,D(a)y]=[D(a^{\ast})x,y]=0$ whence $D(a)y=0$. We proved
that $H_{0}^{\bot}\subset\ker D(L)$.

On the other hand, since $K\subset H_{0}$, then for each $y\in K$ and each
$\varepsilon>0$ there is $z\in D(A)H$ with $\Vert z-y\Vert<\varepsilon$. Note
that $\Vert D(u_{n})z-z\Vert\rightarrow0$ when $n\rightarrow\infty$, because
$D(u_{n})D(a)x=D(u_{n}a)x\rightarrow D(a)x$. Since $D(u_{n})y=0$, we get that
\[
\Vert z\Vert=\lim\Vert D(u_{n})(z-y)\Vert\leq C\Vert D\Vert\varepsilon,
\]
where $C=\sup_{n}\Vert u_{n}\Vert$. Therefore
\[
\Vert y\Vert\leq\Vert z\Vert+\Vert y-z\Vert\leq\varepsilon(1+C\Vert D\Vert).
\]
Since $\varepsilon$ can be arbitrary we conclude that $y=0$. Thus $K=0$ and
$H_{0}$ is non-degenerate.
\end{proof}

\begin{proof}
[The proof of Theorem \ref{grA}]Let now $H$ be a $\Pi_{k}$-space and
$D:L^{1}(G)\rightarrow B(H)$ be a continuous $J$-symmetric representation. It
is known that $L^{1}(G)$ has a bounded approximate identity $\{u_{n}\}$
(moreover $\Vert u_{n}\Vert=1$, for all $n$) so, by Lemma \ref{decomp}, $H$
decomposes in $J$-orthogonal sum of subspaces $H=H_{0}+H_{0}^{\bot}$, where
$H_{0}$ is the essential subspace for $D$.

The algebra $L^{1}(G)$ is an ideal of the *-algebra $M(G)$ of all finite
measures on $G$; we will denote the involution in $M(G)$ by $\mu\mapsto \mu^{\flat}$ and the product by $\mu\ast \nu$. Applying Lemma \ref{ext} to the restriction of $D$ to $H_{0}%
$, we have that there is a representation $\widetilde{D}$ of $M(G)$ on $H_{0}$
extending $D$. To check that $\widetilde{D}$ is $J$-symmetric, it suffices to
check the equality $[\widetilde{D}(\mu)x,y]=[x,\widetilde{D}(\mu^{\flat})y]$,
for $x$ of the form $D(f)z$, where $f\in L^{1}(G)$, $z\in H_{0}$. In this case
we have
\begin{align*}
\lbrack\widetilde{D}(\mu)x,y]  &  =[\widetilde{D}(\mu)D(f)z,y]=[D(\mu\ast 
f)z,y]=[z,D(f^{\flat}\ast\mu^{\flat})y]\\
&  =[z,D(f^{\flat})\widetilde{D}(\mu^{\flat})y]=[D(f)z,\widetilde{D}(\mu^{\flat
})y]\\
&  =[x,\widetilde{D}(\mu^{\flat})y].
\end{align*}
For each $g\in G$, we denote by $\delta_{g}$ the point measure in $g$. Setting
$\pi(g)=\widetilde{D}(\delta_{g})$, we obtain a $J$-unitary representation of
$G$. Indeed, since $(\delta_{g})^{-1}=\delta_{g^{-1}}$, we have
\[
\pi(gh)=\widetilde{D}(\delta_{gh})=\widetilde{D}(\delta_{g}\ast\delta
_{h})=\widetilde{D}(\delta_{g})\widetilde{D}(\delta_{h})=\pi(g)\pi(h),
\]
and
\[
\pi(g)^{\sharp}=(\widetilde{D}(\delta_{g}))^{\sharp}=\widetilde{D}%
(\delta_{g^{-1}})=\pi(g)^{-1}.
\]
Since $\Vert\delta_{g}\Vert=1$,
\[
\Vert\pi(g)\Vert\leq\left\Vert \widetilde{D}\right\Vert ,
\]
so $\pi$ is bounded.

Let us check that the representation $\pi$ is strongly continuous. Since $\pi$ is
bounded, it suffices to verify that the function $g\mapsto\pi(g)x$ is
continuous for $x$ in a dense subset of $H_{0}$. So we may take $x=D(f)y$, for
some $f\in L^{1}(G)$, $y\in H_{0}$. Since the map $g\longmapsto(\delta
_{g}\ast f)(h)=f(g^{-1}h)$ from $G$ to $L^{1}(G)$ is continuous, we get that
\[
\pi(g)x=\pi(g)D(f)y=D(\delta_{g}\ast f)y
\]
continuously depends on $g$.

Applying Corollary \ref{bounded}, we find an invariant dual pair of subspaces
$K_{+},K_{-}$ of $H_{0}$ invariant for all operators $\pi(g)$. To see that
these subspaces are invariant for $D(L^{1}(G))$, let us denote by $W$ the
representation of $L^{1}(G)$ generated by $\pi$:
\[
W(f)=\int_{G}f(g)\pi(g)dg.
\]
Clearly $K_{+}$ and $K_{-}$ are invariant for all operators $W(f)$, and we
have only to show that $W(f)=D(f)$, for all $f\in L^{1}(G)$.

Since $\pi(g)(D(f)x)=D(\delta_{g}\ast f)x$, for all $x\in H$ and $f\in L^{1}(G)$, we
have
\begin{align*}
W(u)D(f)x  &  =\int_{G}u(g)\pi(g)D\left(  f\right)  xdg=\int_{G}u(g)D\left(
\delta_{g}\ast f\right)  xdg\\
&  =D\left(  \int_{G}u(g)(\delta_{g}\ast f)dg\right)  x=D(u\ast f)x\\
&  =D(u)D(f)x
\end{align*}
for each $u\in L^{1}(G)$. Since vectors of the form $D(f)x$ generate $H$, we
conclude that $W(u)=D(u)$.

As we know, the restrictions of all operators $D(f)$, $f\in L^{1}(G)$, to
$H_{0}^{\bot}$ are trivial. So we may choose any dual pair $N_{+},N_{-}$ of
$H_{0}^{\bot}$ and, setting $H_{+}=K_{+}+N_{+}$, $H_{-}=K_{-}+N_{-},$ we will
obtain a dual pair in $H$ invariant for $D(L^{1}(G))$.
\end{proof}

{\bf Acknowledgement}.  The authors are grateful to the Referee for many helpful suggestions and remarks.

\end{document}